\documentclass[a4paper,12pt]{article} 
\usepackage{amsthm}
\usepackage{amsthm, amsmath, amssymb, url,tikz}
\usepackage{hyperref}
\usepackage{pgfplots}
\pgfplotsset{compat=newest}
\usepackage{geometry}
\usepackage{mathrsfs}
\usepackage{caption}
\usepackage{graphicx}
\newtheorem{theorem}{Theorem}[section]
\newtheorem{lemma}[theorem]{Lemma}
\newtheorem{definition}[theorem]{Definition}
\newtheorem{proposition}[theorem]{Proposition}
\newtheorem{corollary}[theorem]{Corollary}
\newtheorem{observation}[theorem]{Observation}

\newtheorem{remark}[theorem]{Remark}

\usepackage[numbers,sort&compress]{natbib}

\title{Characterization of the structure of $k$-edge-maximal graphs}
\author{Zheng-Jiang Xia$^{\,\rm a}$, Hong-Jian Lai$^{\,\rm b,c}$, Jian Lu$^{\,\rm d}$, Zhen-Mu Hong$^{\,\rm d}$\thanks{Corresponding author. E-mail addresses: xzj@mail.ustc.edu.cn (Z.-J. Xia), hjlai2015@hotmail.com (H.-J. Lai), lujianmath@163.com (J. Lu), zmhong@mail.ustc.edu.cn (Z.-M. Hong).}\\
{\footnotesize$^{\rm a}$School of Finance, Anhui University of Finance \& Economics, Bengbu 233030, China}\\
{\footnotesize$^{\rm b}$School of Mathematics and Systems Science, Guangdong Polytechnic Normal University,} \\ {\footnotesize Guangzhou 510665, China}\\
{\footnotesize$^{\rm c}$Department of Mathematics, West Virginia University, Morgantown, WV 26506, USA}\\
{\footnotesize$^{\rm d}$School of Statistics and Applied Mathematics, Anhui University of Finance \& Economics,} \\{\footnotesize Bengbu 233030, China}}

\date{}

\begin{document}

\maketitle

\parbox{14.5cm}{
\begin{center}
\textbf{Abstract}
\end{center}

Let $\kappa^{\prime}(G)$ be the edge-connectivity of the graph $G$. The \textit{strength} of $G$, denoted by $\overline{\kappa}^{\prime}(G)$, is the maximum edge-connectivity of its subgraphs. A simple graph $G$ is called $k$-\textit{edge-maximal} if $\overline{\kappa}^{\prime}(G) \leq k$ but for any edge $e$ not in $G$, $\overline{\kappa}^{\prime}(G+e) \geq k+1$. In this paper, we propose the concepts of kernel and closure of a graph and discuss the properties of closure. Utilizing these properties, 
we present the necessary and sufficient condition for a graph to be $k$-edge-maximal,
which refines the results in [J. Graph Theory 14 (1990) 187--197],
and prove that there exists a $k$-edge-maximal graph of order $n$ with $m$ edges if and only if $m=(n-1)k-\binom{k}{2}r$, for some integer $r$ with $1\leq r\leq \left\lfloor \frac{n}{k+2}\right\rfloor$. 
Furthermore, we characterize the structure of $k$-edge-maximal graphs with a given number of edges.

\vskip0.4cm
\noindent{\bf Keywords:} Edge-connectivity; Strength; $k$-maximal graph; Closure

\vskip0.4cm \noindent {\bf AMS Subject Classification: }\ 05C35, 05C40, 05C75
}

\section{Introduction}

We consider finite and simple graphs in this paper. Undefined notation and terminologies will follow Bondy and Murty \cite{bm08}.
Let $G^c$ denote the complement of a simple graph $G$. If $X\subseteq E(G^c)$, then $G+X$ is the simple graph with vertex-set $V(G)$ and edge-set $E(G)\cup X$. We will use $G+e$ for $G+\{e\}$. Denote by $H \subseteq G$ if $H$ is a subgraph of $G$ and by $H \cong G$ if $H$ is isomorphic to $G$. If $X$ is a subset of $V(G)$ or of $E(G)$, then $G[X]$ denotes the subgraph of $G$ induced by $X$. Let $d_G(v)$ denote the degree of a vertex $v$ of $G$ and let $\delta(G)$ denote the minimum degree of $G$.  An \textit{edge-cut} of a graph $G$ is an edge subset whose removal increases the number of components of $G$. An edge-cut of size $k$ is called a $k$-\textit{edge-cut}. 
The graph $G$ has \textit{edge-connectivity} $k$ if $G$ contains a $k$-edge-cut
but no smaller edge-cuts. The edge-connectivity of $G$ is denoted by $\kappa'(G)$. The graph $G$ is \textit{$k$-edge-connected} if $\kappa'(G)\geq k$. 
Throughout this paper, $\mathbb{N}$ denotes the set of all positive integers. For any $k\in \mathbb{N}$, we define $\binom{k}{2}=\frac{1}{2}k(k-1)$ and so $\binom{1}{2}=0$. The \textit{join} of two graphs $G_1$ and $G_2$, denoted by $G_1 \vee G_2$, is a graph with 
$V(G_2 \vee G_2)=V(G_1) \cup V(G_2)$ and
$$
E(G_1 \vee G_2)=E(G_1) \cup E(G_2) \cup\{xy: x \in V(G_1) \text { and } y \in V(G_2)\} .
$$

In \cite{mat72}, Matula formally defined the strength of a graph $G$, denoted by $\overline{\kappa}^{\prime}(G)$, to be the maximum edge-connectivity of its subgraphs. More precisely, we have
$$
\overline{\kappa}^{\prime}(G)=\max _{H \subseteq G} \kappa^{\prime}(H),
$$
where the optimum are taken over all subgraphs $H$ of $G$. Mader \cite{mad71} investigated the extremal size of a simple graph with bounded strength. For a positive integer $k$, a simple graph $G$ is called \textit{$k$-edge-maximal} if $\overline{\kappa}^{\prime}(G) \leq k$ but $\overline{\kappa}^{\prime}(G+e) \geq k+1$, for any edge $e \notin E(G)$. 

The maximum subgraph edge-connectivity problem was earlier investigated by Mader \cite{mad71} and Matula \cite{mat72}, and has been intensively studied by many researchers, as found in 
\cite{allx17,lai90,lz94,lflx16,mat68,mat69,mat76,tlm19,tml22,tlmx23,wlyl26,xtl20,xlz20,xltx22}, among others. For more references, one can see the survey \cite{tm26} and the references therein.

For integers $n$ and $k$ with $n \geq k+1$, we define
$$
F(n, k)=\max \{|E(G)|: G \text{ is a $k$-edge-maximal graph of order } n\}, 
$$
$$
f(n, k)=\min \{|E(G)|: G \text{ is a $k$-edge-maximal graph of order } n\},
$$
$$
\mathcal{G}(F ; n, k)=\{G: G \text{ is a $k$-edge-maximal graph of order $n$ and } |E(G)|=F(n, k)\}, \text{ and } 
$$
$$
\mathcal{G}(f ; n, k)=\{G: G \text{ is a $k$-edge-maximal graph of order $n$ and } |E(G)|=f(n, k)\}.
$$

Mader \cite{mad71} and Lai \cite{lai90} determined the parameters $F(n,k)$ and $f(n,k)$, respectively.

\begin{theorem}\label{thm:mad and lai}
Let $k\geq 1$ be an integer, and let $G$ be a $k$-edge-maximal graph of order $n$. Each of the following holds.
\begin{itemize}
\item[$(i)$] {\rm (Mader \cite{mad71})} For $n \geq k+1$, $|E(G)| \leq F(n, k)=(n-1) k-\binom{k}{2}$. 

\item[$(ii)$] {\rm (Lai \cite{lai90})} For $n\geq k+2$, $|E(G)| \geq f(n, k)=(n-1) k-\binom{k}{2}\left\lfloor \frac{n}{k+2}\right\rfloor$. 
\end{itemize}
\end{theorem}

Let $k$ be an integer and let $H_1$ and $H_2$ be two graphs with disjoint vertex sets and with $\max \left\{|V(H_1)|,|V(H_2)|\right\} \geq k+1$. A \textit{$k$-edge-join} of $H_1$ and $H_2$ is a simple graph obtained from the disjoint union of $H_1$ and $H_2$ by adding $k$ new edges $e_1, e_1, \ldots, e_k$ between $H_1$ and $H_2$ such that each $e_i$ is incident to a vertex of $V(H_1)$ and a vertex of $V(H_2)$. Denote by $\left[H_1, H_2\right]_k$ the set of all $k$-edge joins of $H_1$ and $H_2$. 
For notational convenience, we also use $[H_1,H_2]_k$ to denote any graph in the family $[H_1,H_2]_k$. The following theorem, published in \cite{lai90}, contains a flaw.

\begin{theorem}[Lai \cite{lai90}]\label{lemma-lai3}
Let $k\geq 1$ be an integer, let $H_1$ be a $k$-edge-maximal graph and let $H_2$ be either a $K_1$ or a $k$-edge-maximal graph. Then all graphs in $\left[H_1, H_2\right]_k$ are $k$-edge-maximal.
\end{theorem}

It is straightforward to verify that none of the graphs in $[K_3,K_3]_2$ are $2$-edge-maximal, which are counterexamples to  Theorem~\ref{lemma-lai3}. 
This implies the $k$-edge-join of two $k$-edge-maximal graphs
does not always result in a $k$-edge-maximal graph.
A natural problem arises: under what conditions can the graph obtained by the $k$-edge-join operation be $k$-edge-maximal?  This problem motivates the current research. In this paper, we establish the necessary and sufficient condition for $[H_1,H_2]_k$ to be $k$-edge-maximal, where $H_1$ and $H_2$ are arbitrary graphs. In order to present our main results, we first propose the concepts of kernel and closure of $k$-edge-maximal graphs.

\begin{definition}\label{def-kernel}{\rm 
Let $G$ and $H$ be two $k$-edge-maximal graphs, where $H$ is a subgraph of  $G$. The graph $H$ is a \textit{$k$-kernel} of $G$, if $G$ can be obtained from $H$ via some $k$-edge-join operations, or equivalently, $H$ can be obtained from $G$ by recursively deleting a $k$-edge-cut of a remaining component, in the sense that $H$ is one of the components after deleting such a series of edge-cuts from $G$. When no confusion arises, we abbreviate a $k$-kernel as a \textit{kernel}.
}
\end{definition}

\begin{definition}\label{def-closure}{\rm 
Let $G$ be a $k$-edge-maximal graph and let $H$ be a subgraph of $G$. A \textit{$k$-closure} of $H$ is a minimal $k$-kernel of $G$ that contains $H$ as a subgraph. When no confusion arises, we abbreviate a $k$-closure as a \textit{closure}.}
\end{definition}

Investigating the properties of the closures of $k$-edge-maximal graphs, we establish our main results as follows. 

\begin{theorem}\label{thm:main result}
Let $H_1$ and $H_2$ be two graphs, and let $G\in [H_1,H_2]_k$ be a graph of order $n> k+1>2$. Denote $V_i$ as the set of vertices in $H_i$ that are incident with the edges joining $H_1$ and $H_2$ in $G$ for each $i\in \{1, 2\}$. 
Then $G$ is $k$-edge-maximal if and only if 
one of the following holds:
\begin{itemize}
\item[$(i)$] $H_i\cong K_1$  and $H_{3-i}$ is a $k$-edge-maximal graph;

\item[$(ii)$] Both $H_1$ and $H_2$ are $k$-edge-maximal graphs with at least $k+2$ vertices, and there is a closure of $H_i[V_i]$ in $H_i$ not isomorphic to $K_{k+1}$ for each $i\in \{1, 2\}$.
\end{itemize}
\end{theorem}

\begin{theorem}\label{thm:main result 2}
For $n> k+1>2$, there exists a $k$-edge-maximal graph of order $n$ with $m$ edges if and only if 
$$
m\in \left\{(n-1)k-\binom{k}{2}r: 1\leq r\leq\left\lfloor \frac{n}{k+2}\right\rfloor \text{ and } r\in \mathbb{N}\right\}.
$$
\end{theorem}

Theorem \ref {thm:mad and lai} by Mader \cite {mad71} and Lai \cite {lai90} follows immediately from Theorem \ref{thm:main result 2}, by
substituting $r=1$ and $r=\left\lfloor \frac{n}{k+2}\right\rfloor$. 

Let $H_1$ and $H_2$ be either 
a $k$-edge-maximal graph or $K_1$. A \textit{proper $k$-edge-join} of $H_1$ and $H_2$ is a $k$-edge-join satisfying the condition stated in Theorem~\ref{thm:main result}, that is, if the resulting graph $[H_1,H_2]_k$ is $k$-edge-maximal. The following theorem further characterizes the structure of $k$-edge-maximal graphs with a given number of edges.

\begin{theorem}\label{thm:main result 3}
Let $G$  be a $k$-edge-maximal graph of order $n>k+1>2$. Then $G$ has  
$$
m = (n-1)k - \binom{k}{2}r
$$  
edges, where $ 1 \leq r \leq \left\lfloor \frac{n}{k+2} \right\rfloor $ and $ r \in \mathbb{N} $,  
if and only if $G$ can be obtained from the disjoint union of $r$ copies of $K_k \vee 2K_1$ and $(n-(k+2)r)$ copies of $K_1$, that is
$$
r(K_k \vee 2K_1) \;\cup\; (n-(k+2)r)K_1, 
$$  
by a sequence of proper $k$-edge-joins.
\end{theorem}

In the next section, we discuss properties of a $k$-edge-maximal graph and the closure of its subgraphs. The proofs of main results will be presented in the last section.

\section{Preliminary}

In this section, we give a necessary condition for a graph to be $k$-edge-maximal and several properties of the closure in $k$-edge-maximal graphs.

\begin{lemma}[Lai \cite{lai90}]\label{lemma-lai1} 
If $G$ is a $k$-edge-maximal graph of order at least $k+2$, then $\overline{\kappa}^{\prime}(G)=\kappa^{\prime}(G)=k$.
\end{lemma}

The following lemma appears in \cite{lai90}. For completeness, we include a clearer and detailed proof here.

\begin{lemma}[Lai \cite{lai90}]\label{lemma-lai2}
Let $G$ be a $k$-edge-maximal graph of order $n$, where $n>k+ 1>2$. Suppose that $X$ is a $k$-edge-cut of $G$, and let $H_1$ and $H_2$ denote the two components of $G-X$. Then one of the following holds:
\begin{itemize}
\item[$(i)$] $H_i\cong K_1$ and $H_{3-i}$ is a $k$-edge-maximal graph;

\item[$(ii)$] Both $H_1$ and $H_2$ are  $k$-edge-maximal graphs with at least $k+2$ vertices.
\end{itemize}
\end{lemma}

\begin{proof}
Suppose first that $H_2 \cong K_1$. If $H_1$ is complete, then since $n>k+1$, $\left|V(H_1)\right| \geq k+1$. Since $\overline{\kappa}^{\prime}(G)=\kappa^{\prime}(G)=k$ and $H_1$ is complete, $H_1$ has order at most $k+1$. Thus $H_1 \cong K_{k+1}$, which is a $k$-edge-maximal graph. Now we assume that $H_1$ is not complete. Let $e \in E(H_1^c)$. Since $G$ is $k$-edge-maximal, there is a subgraph $H\subseteq G+e$ such that $\kappa^{\prime}(H) \geq k+1$. 
Since $H$ is simple with $\delta(H) \geq \kappa^{\prime}(H) \geq k+1$,
we have $H\subseteq H_1$ and
$\left|V(H_1)\right| \geq|V(H)| \geq k+2$. Hence $H_1$ is $k$-edge-maximal.

Similarly, the lemma will follow if $H_1 \cong K_1$. This proves Lemma \ref{lemma-lai2} (i).

Now we assume that both $H_1$ and $H_2$ have at least two vertices.

\textbf{Case 1.} Suppose that one of the $G_i$'s is complete, say
$H_1 \cong K_t$, for some $t \geq 2$. We shall derive a contradiction. 

Since $G$ is $k$-edge-maximal, $\kappa^{\prime}(G) \leq k$ and so $t \leq k+1$. Thus $2 \leq t \leq k+1$.

If $t=k+1$, then there is a vertex $u\in V(H_1)$ which is not adjacent to the vertices of $H_2$. As $d_G(u)=k$, we have $G-u$ is $k$-edge-maximal by Lemma \ref{lemma-lai2} (i), and $G-u\in [K_k,H_2]_k$. Thus we need only consider the graph $G-u$ instead of $G$. Therefore, we may assume 
\begin{equation}\label{eq:eq1}
 2\leq t\leq k.   
\end{equation}

We claim that for any two vertices $v_1$ and $v_2$ of $G$, if $d_G(v_1)=d_G(v_2)=k$, then $v_1v_2 \notin E(G)$. To the contrary, if $v_1v_2\in E(G)$, then $d_{G-v_1}(v_2)=k-1$. By Lemma \ref{lemma-lai2} (i), $G-v_1$ is also $k$-edge-maximal, which yields  $d_{G-v_1}(v_2)\geq \kappa'(G-v_1)=k$, a contradiction. Thus, there exists
at most one vertex $u$ of $H_1\cong K_t$ with $d_G(u)=k$, and so $d_G(v)\geq k+1$ for each $v\in V(H_1)\setminus \{u\}$. 
This fact, $H_1\cong K_t$, and $|X|=k$ give
\begin{equation*}
t(t-1)+k=\sum_{v\in V(H_1)}d_G(v)=d_G(u)+\sum_{v\in V(H_1)\setminus \{u\}}d_G(v)\geq k+(k+1)(t-1). 
\end{equation*}
Thus, 
$$ 
(t-1)(t-k-1)\geq 0, 
$$
implying that $t\leq 1$ or $t\geq k+1$, which contradicts \eqref{eq:eq1}.
This concludes Case 1.

\textbf{Case 2.} $H_1$ is not complete.

For any edge $e \in E(H_1^c) \subseteq E(G^c)$, $G+e$ has a subgraph $L$ with $\kappa^{\prime}(L) \geq k+1$. Since $|X| \leq k$, $X\cap E(L)=\emptyset$ which means $L$ is a subgraph of $H_1$. Note that $\left|V(H_1)\right| \geq|V(L)| \geq k+1$. Hence $H_1$ is a $k$-edge-maximal if we can show that $\overline{\kappa}^{\prime}(H_1) \leq k$. By Lemma \ref{lemma-lai1}, $ \overline{\kappa}^{\prime}(H_1) \leq \overline{\kappa}^{\prime}(G)=k$, and so $H_1$ is $k$-edge-maximal. Since $H_1$ is not complete, $H_1$ is not isomorphic to $K_{k+1}$ and $H_1$ has at least $k+2$ vertices.

Similarly, $H_2$ is a $k$-edge-maximal graph with at least $k+2$ vertices.
\end{proof}

\begin{remark}{\rm
Lemma~\ref{lemma-lai2} provides a necessary condition for a graph to be $k$-edge-maximal. For any given graph, delete one $k$-edge-cut at a time and check if the resulting components are isolated vertices or $k$-edge-maximal graphs. Repeat this process. If any component with at least $k+1$ vertices is not $k$-edge-maximal, then the original graph is not $k$-edge-maximal.}
\end{remark}

By Lemma \ref{lemma-lai2}, we get the following corollaries.

\begin{corollary}\label{coro:2.4}
Let $k$ and $t$ be two integers with $k\geq 2$ and $2\leq t\leq k+1$. Let $G$ be a $k$-edge-maximal graph. If 
$G'\in [K_{t},G]_k$, then $G'$ is not $k$-edge-maximal. 
\end{corollary}

\begin{corollary}\label{coro:2.5}
Let $G$ be a $k$-edge-maximal graph of order $n> k+1>2$. If $v$ is a vertex of $G$ of degree $k$, then $G-v$ is also $k$-edge-maximal. 
\end{corollary}

Now, we present several properties of the closure, which are of interest in its own and will be used in the proof of Theorem \ref{thm:main result}.

Recalling the definitions of kernel and closure in Definition \ref{def-kernel} and Definition \ref{def-closure}, we have the following observation.

\begin{observation}\label{obs}
Each of the following holds.
\begin{itemize}
    \item[$(i)$] Every $k$-edge-maximal graph is a kernel of itself;
    \item[$(ii)$] $K_{k+1}$ is a smallest kernel of any $k$-edge-maximal graph;
    \item[$(iii)$] Every subgraph of a $k$-edge-maximal graph has a closure;
    \item[$(iv)$] A kernel of a $k$-edge-maximal graph is $k$-edge-maximal;
    \item[$(v)$] A $k$-edge-maximal subgraph of a $k$-edge-maximal graph $G$ is not necessarily a kernel of $G$.
\end{itemize}
\end{observation}

We explain Observation \ref{obs} (v) by giving more details. Let $H$ be a $k$-edge-maximal subgraph of the $k$-edge-maximal graph $G$. By Definition~\ref{def-kernel}, if $H$ cannot be obtained from $G$ by deleting a series of $k$-edge-cuts, then $H$ is not a kernel of $G$. For example, the triangle containing $v$ in Figure \ref{fig:fig0} is $2$-edge-maximal, but it is not a kernel of the $k$-edge-maximal graph $G_2$.

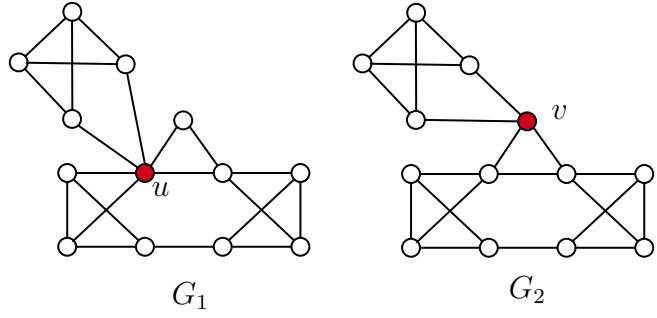
\begin{figure}[htbp]
    \centering

\tikzset{every picture/.style={line width=0.75pt}} 

\begin{tikzpicture}[x=0.75pt,y=0.75pt,yscale=-1,xscale=1]

\draw   (193.47,143.13) .. controls (193.47,140.65) and (195.48,138.63) .. (197.97,138.63) .. controls (200.45,138.63) and (202.47,140.65) .. (202.47,143.13) .. controls (202.47,145.62) and (200.45,147.63) .. (197.97,147.63) .. controls (195.48,147.63) and (193.47,145.62) .. (193.47,143.13) -- cycle ;
\draw   (193.47,180.13) .. controls (193.47,177.65) and (195.48,175.63) .. (197.97,175.63) .. controls (200.45,175.63) and (202.47,177.65) .. (202.47,180.13) .. controls (202.47,182.62) and (200.45,184.63) .. (197.97,184.63) .. controls (195.48,184.63) and (193.47,182.62) .. (193.47,180.13) -- cycle ;
\draw  [fill={rgb, 255:red, 208; green, 2; blue, 27 }  ,fill opacity=1 ] (232.47,143.13) .. controls (232.47,140.65) and (234.48,138.63) .. (236.97,138.63) .. controls (239.45,138.63) and (241.47,140.65) .. (241.47,143.13) .. controls (241.47,145.62) and (239.45,147.63) .. (236.97,147.63) .. controls (234.48,147.63) and (232.47,145.62) .. (232.47,143.13) -- cycle ;
\draw   (232.47,180.13) .. controls (232.47,177.65) and (234.48,175.63) .. (236.97,175.63) .. controls (239.45,175.63) and (241.47,177.65) .. (241.47,180.13) .. controls (241.47,182.62) and (239.45,184.63) .. (236.97,184.63) .. controls (234.48,184.63) and (232.47,182.62) .. (232.47,180.13) -- cycle ;
\draw    (202.47,143.13) -- (232.47,143.13) ;
\draw    (202.47,180.13) -- (232.47,180.13) ;
\draw    (197.97,147.63) -- (197.97,175.63) ;
\draw    (201.12,146.3) -- (233.82,176.97) ;
\draw    (233.62,146.55) -- (200.62,177.55) ;
\draw   (271.47,143.13) .. controls (271.47,140.65) and (273.48,138.63) .. (275.97,138.63) .. controls (278.45,138.63) and (280.47,140.65) .. (280.47,143.13) .. controls (280.47,145.62) and (278.45,147.63) .. (275.97,147.63) .. controls (273.48,147.63) and (271.47,145.62) .. (271.47,143.13) -- cycle ;
\draw   (271.47,180.13) .. controls (271.47,177.65) and (273.48,175.63) .. (275.97,175.63) .. controls (278.45,175.63) and (280.47,177.65) .. (280.47,180.13) .. controls (280.47,182.62) and (278.45,184.63) .. (275.97,184.63) .. controls (273.48,184.63) and (271.47,182.62) .. (271.47,180.13) -- cycle ;
\draw   (310.47,143.13) .. controls (310.47,140.65) and (312.48,138.63) .. (314.97,138.63) .. controls (317.45,138.63) and (319.47,140.65) .. (319.47,143.13) .. controls (319.47,145.62) and (317.45,147.63) .. (314.97,147.63) .. controls (312.48,147.63) and (310.47,145.62) .. (310.47,143.13) -- cycle ;
\draw   (310.47,180.13) .. controls (310.47,177.65) and (312.48,175.63) .. (314.97,175.63) .. controls (317.45,175.63) and (319.47,177.65) .. (319.47,180.13) .. controls (319.47,182.62) and (317.45,184.63) .. (314.97,184.63) .. controls (312.48,184.63) and (310.47,182.62) .. (310.47,180.13) -- cycle ;
\draw    (280.47,143.13) -- (310.47,143.13) ;
\draw    (280.47,180.13) -- (310.47,180.13) ;
\draw    (279.12,146.3) -- (311.82,176.97) ;
\draw    (311.62,146.55) -- (279.12,176.8) ;
\draw    (314.97,147.63) -- (314.97,175.63) ;
\draw    (241.47,143.13) -- (271.47,143.13) ;
\draw    (241.47,180.13) -- (271.47,180.13) ;
\draw   (251.67,116.73) .. controls (251.67,114.25) and (253.68,112.23) .. (256.17,112.23) .. controls (258.65,112.23) and (260.67,114.25) .. (260.67,116.73) .. controls (260.67,119.22) and (258.65,121.23) .. (256.17,121.23) .. controls (253.68,121.23) and (251.67,119.22) .. (251.67,116.73) -- cycle ;
\draw    (253,120) -- (240,139.5) ;
\draw    (259.5,120) -- (273.5,139.5) ;
\draw   (197.44,58.98) .. controls (199.25,57.28) and (202.09,57.36) .. (203.8,59.16) .. controls (205.51,60.97) and (205.43,63.82) .. (203.62,65.53) .. controls (201.81,67.23) and (198.96,67.15) .. (197.26,65.34) .. controls (195.55,63.54) and (195.63,60.69) .. (197.44,58.98) -- cycle ;
\draw   (170.54,84.39) .. controls (172.35,82.68) and (175.19,82.76) .. (176.9,84.57) .. controls (178.61,86.37) and (178.52,89.22) .. (176.72,90.93) .. controls (174.91,92.64) and (172.06,92.55) .. (170.36,90.75) .. controls (168.65,88.94) and (168.73,86.09) .. (170.54,84.39) -- cycle ;
\draw   (224.22,85.34) .. controls (226.02,83.63) and (228.87,83.71) .. (230.58,85.52) .. controls (232.28,87.33) and (232.2,90.17) .. (230.4,91.88) .. controls (228.59,93.59) and (225.74,93.51) .. (224.03,91.7) .. controls (222.33,89.89) and (222.41,87.04) .. (224.22,85.34) -- cycle ;
\draw   (197.31,112.74) .. controls (199.12,111.03) and (201.97,111.12) .. (203.68,112.92) .. controls (205.38,114.73) and (205.3,117.58) .. (203.49,119.28) .. controls (201.69,120.99) and (198.84,120.91) .. (197.13,119.1) .. controls (195.43,117.3) and (195.51,114.45) .. (197.31,112.74) -- cycle ;
\draw    (203.62,65.53) -- (224.22,85.34) ;
\draw    (176.72,90.93) -- (197.31,112.74) ;
\draw    (197.26,65.34) -- (176.9,84.57) ;
\draw    (200.39,66.72) -- (200.54,111.55) ;
\draw    (223.5,88.5) -- (177.33,87.81) ;
\draw    (203.49,119.28) -- (233.5,140.5) ;
\draw    (228,93) -- (236.97,138.63) ;
\draw   (365.97,143.63) .. controls (365.97,141.15) and (367.98,139.13) .. (370.47,139.13) .. controls (372.95,139.13) and (374.97,141.15) .. (374.97,143.63) .. controls (374.97,146.12) and (372.95,148.13) .. (370.47,148.13) .. controls (367.98,148.13) and (365.97,146.12) .. (365.97,143.63) -- cycle ;
\draw   (365.97,180.63) .. controls (365.97,178.15) and (367.98,176.13) .. (370.47,176.13) .. controls (372.95,176.13) and (374.97,178.15) .. (374.97,180.63) .. controls (374.97,183.12) and (372.95,185.13) .. (370.47,185.13) .. controls (367.98,185.13) and (365.97,183.12) .. (365.97,180.63) -- cycle ;
\draw   (404.97,143.63) .. controls (404.97,141.15) and (406.98,139.13) .. (409.47,139.13) .. controls (411.95,139.13) and (413.97,141.15) .. (413.97,143.63) .. controls (413.97,146.12) and (411.95,148.13) .. (409.47,148.13) .. controls (406.98,148.13) and (404.97,146.12) .. (404.97,143.63) -- cycle ;
\draw   (404.97,180.63) .. controls (404.97,178.15) and (406.98,176.13) .. (409.47,176.13) .. controls (411.95,176.13) and (413.97,178.15) .. (413.97,180.63) .. controls (413.97,183.12) and (411.95,185.13) .. (409.47,185.13) .. controls (406.98,185.13) and (404.97,183.12) .. (404.97,180.63) -- cycle ;
\draw    (374.97,143.63) -- (404.97,143.63) ;
\draw    (374.97,180.63) -- (404.97,180.63) ;
\draw    (370.47,148.13) -- (370.47,176.13) ;
\draw    (373.62,146.8) -- (406.32,177.47) ;
\draw    (406.12,147.05) -- (373.12,178.05) ;
\draw   (443.97,143.63) .. controls (443.97,141.15) and (445.98,139.13) .. (448.47,139.13) .. controls (450.95,139.13) and (452.97,141.15) .. (452.97,143.63) .. controls (452.97,146.12) and (450.95,148.13) .. (448.47,148.13) .. controls (445.98,148.13) and (443.97,146.12) .. (443.97,143.63) -- cycle ;
\draw   (443.97,180.63) .. controls (443.97,178.15) and (445.98,176.13) .. (448.47,176.13) .. controls (450.95,176.13) and (452.97,178.15) .. (452.97,180.63) .. controls (452.97,183.12) and (450.95,185.13) .. (448.47,185.13) .. controls (445.98,185.13) and (443.97,183.12) .. (443.97,180.63) -- cycle ;
\draw   (482.97,143.63) .. controls (482.97,141.15) and (484.98,139.13) .. (487.47,139.13) .. controls (489.95,139.13) and (491.97,141.15) .. (491.97,143.63) .. controls (491.97,146.12) and (489.95,148.13) .. (487.47,148.13) .. controls (484.98,148.13) and (482.97,146.12) .. (482.97,143.63) -- cycle ;
\draw   (482.97,180.63) .. controls (482.97,178.15) and (484.98,176.13) .. (487.47,176.13) .. controls (489.95,176.13) and (491.97,178.15) .. (491.97,180.63) .. controls (491.97,183.12) and (489.95,185.13) .. (487.47,185.13) .. controls (484.98,185.13) and (482.97,183.12) .. (482.97,180.63) -- cycle ;
\draw    (452.97,143.63) -- (482.97,143.63) ;
\draw    (452.97,180.63) -- (482.97,180.63) ;
\draw    (451.62,146.8) -- (484.32,177.47) ;
\draw    (484.12,147.05) -- (451.62,177.3) ;
\draw    (487.47,148.13) -- (487.47,176.13) ;
\draw    (413.97,143.63) -- (443.97,143.63) ;
\draw    (413.97,180.63) -- (443.97,180.63) ;
\draw  [fill={rgb, 255:red, 208; green, 2; blue, 27 }  ,fill opacity=1 ] (424.17,117.23) .. controls (424.17,114.75) and (426.18,112.73) .. (428.67,112.73) .. controls (431.15,112.73) and (433.17,114.75) .. (433.17,117.23) .. controls (433.17,119.72) and (431.15,121.73) .. (428.67,121.73) .. controls (426.18,121.73) and (424.17,119.72) .. (424.17,117.23) -- cycle ;
\draw    (425.5,120.5) -- (412.5,140) ;
\draw    (432,120.5) -- (446,140) ;
\draw   (369.94,59.48) .. controls (371.75,57.78) and (374.59,57.86) .. (376.3,59.66) .. controls (378.01,61.47) and (377.93,64.32) .. (376.12,66.03) .. controls (374.31,67.73) and (371.46,67.65) .. (369.76,65.84) .. controls (368.05,64.04) and (368.13,61.19) .. (369.94,59.48) -- cycle ;
\draw   (343.04,84.89) .. controls (344.85,83.18) and (347.69,83.26) .. (349.4,85.07) .. controls (351.11,86.87) and (351.02,89.72) .. (349.22,91.43) .. controls (347.41,93.14) and (344.56,93.05) .. (342.86,91.25) .. controls (341.15,89.44) and (341.23,86.59) .. (343.04,84.89) -- cycle ;
\draw   (396.72,85.84) .. controls (398.52,84.13) and (401.37,84.21) .. (403.08,86.02) .. controls (404.78,87.83) and (404.7,90.67) .. (402.9,92.38) .. controls (401.09,94.09) and (398.24,94.01) .. (396.53,92.2) .. controls (394.83,90.39) and (394.91,87.54) .. (396.72,85.84) -- cycle ;
\draw   (369.81,113.24) .. controls (371.62,111.53) and (374.47,111.62) .. (376.18,113.42) .. controls (377.88,115.23) and (377.8,118.08) .. (375.99,119.78) .. controls (374.19,121.49) and (371.34,121.41) .. (369.63,119.6) .. controls (367.93,117.8) and (368.01,114.95) .. (369.81,113.24) -- cycle ;
\draw    (376.12,66.03) -- (396.72,85.84) ;
\draw    (349.22,91.43) -- (369.81,113.24) ;
\draw    (369.76,65.84) -- (349.4,85.07) ;
\draw    (372.89,67.22) -- (373.04,112.05) ;
\draw    (396,89) -- (349.83,88.31) ;
\draw    (377,116.5) -- (424.17,117.23) ;
\draw    (402.9,92.38) -- (425.5,114) ;

\draw (238.97,146.53) node [anchor=north west][inner sep=0.75pt]    {$u$};
\draw (439.47,107.53) node [anchor=north west][inner sep=0.75pt]    {$v$};
\draw (248.97,196.53) node [anchor=north west][inner sep=0.75pt]    {$G_{1}$};
\draw (417.97,193.53) node [anchor=north west][inner sep=0.75pt]    {$G_{2}$};

\end{tikzpicture}

    \caption{The non-$2$-edge-maximal graph $G_1$ and the $2$-edge-maximal graph $G_2$}
    \label{fig:fig0}
\end{figure}

\vspace{0.5em}
Let $H$ be a subgraph of $G$. We say that an edge-cut $X$ of $G$ \textit{separates} the vertices of $H$, if the vertices of $H$ lie in distinct components of $G-X$. The following two propositions reveal the property of $k$-edge-cuts of a closure.

\begin{proposition}\label{prop:2.7}
Let $G$ be a $k$-edge-maximal graph of order $n> k+1>2$ and $H$ be a subgraph of $G$ with $|V(H)|\geq 2$. Let $C$ be a closure of $H$ in $G$. If $C\ncong K_{k+1}$, then every $k$-edge-cut of $C$ separates the vertices of $H$.
\end{proposition}

\begin{proof}
By Definition \ref{def-closure}, $C$ is $k$-edge-maximal.
Hence, $\kappa'(C)=k$ by Lemma \ref{lemma-lai1}. 
Suppose to the contrary that there exists a 
$k$-edge-cut $X$ of $C$ such that $V(H)$ belongs to one component $G_1$ of $C-X$. As $|V(G_1)|\geq |V(H)|\geq 2$, $G_1$ is also $k$-edge-maximal by Lemma \ref{lemma-lai2}. Thus, $G_1$ is a kernel of $G$ containing $H$ as a subgraph with fewer vertices than $C$, which contradicts the minimality of $C$. The result follows.
\end{proof}

\begin{proposition}\label{prop:2.8}
Let $G$ be a $k$-edge-maximal graph, where $k\geq2$, and let $v$ be a vertex of $G$. If $C$ is a closure of $v$ in $G$, then
$$
d_{C}(v)=k.
$$
Moreover, if $C\ncong K_{k+1}$, then
the trivial edge-cut incident to $v$ is the unique $k$-edge-cut of $C$.
\end{proposition}

\begin{proof}
If $C\cong K_{k+1}$, then the result follows.
Now, assume that $C\ncong K_{k+1}$.
By the definition of closure, $C$ is $k$-edge-maximal. Let $X$ be a $k$-edge-cut of $C$. If $X$ is not the trivial edge-cut incident to $v$, then by Corollary \ref{coro:2.4} the component of $C-X$ containing $v$ would be a smaller kernel of $G$, which contradicts the minimality of $C$. Therefore, $X$ must be a trivial edge-cut and $d_{C}(v)=k$.
\end{proof}

The lemmas below characterize the properties of a closure.  

\begin{lemma}\label{lem:2.9}
Let $G$ be a $k$-edge-maximal graph and $H$ be a subgraph of $G$ with $k\geq 2$ and $|V(H)|\geq 2$. Denote by $\mathcal{C}$ the set of all closures of $H$ in $G$. If there is a closure $C \in \mathcal{C}$ not isomorphic to $ K_{k+1}$, then $|\mathcal{C}| = 1$.
\end{lemma}

\begin{proof}
We prove this by contradiction. 
Suppose that $C_1$ and $C_2$ are two closures of $H$ in $G$ such that $C_1\ncong K_{k+1}$. 

Then, by Proposition~\ref{prop:2.7}, every $k$-edge-cut of $C_1$ separates the vertices of $H$. 
We note that $C_2$ is a closure of $H$, and thus $C_2$ can be derived from $G$ by deleting a series of $k$-edge-cuts.
Let $G_1=G$, and let $G_{i+1}$ be the component of $G_{i}-X_{i}$ containing $C_2$ for $1\leq i\leq s-1$, with $G_s=C_2$, where $X_i$ is a $k$-edge-cut of $G_i$. 
Let $p$ be the least integer such that $X_p\cap E(C_1)\neq \emptyset$. Then $C_1$ is a subgraph of $G_p$ and $X_p$ is a $k$-edge-cut of $G_p$.

Without loss of generality, assume $e=xy\in X_p\cap E(C_1)$. Then $X_p$ must separate the vertices $x$ and $y$ in $C_1$. Note that $\kappa'(C_1)=k$; thus, separating $x$ and $y$ requires at least $k$ edges of $C_1$. Thus $X_p$ is also a $k$-edge-cut of $C_1$. By Proposition~\ref{prop:2.7}, $X_p$ separates the vertices of $H$, and so the vertices of $H$ cannot be contained in the same component of $G_p-X_p$. Since $C_2$ contains $H$ as a subgraph, $C_2$ cannot be contained in the same component of $G_p-X_p$, which is a contradiction to $G_s=C_2$.
\end{proof}

\begin{lemma}\label{lem:2.10}
Let $G$ be a $k$-edge-maximal graph, where $k\geq2$, and let $v$ be a vertex of $G$. Denote by $\mathcal{C}$ the set of all closures of $v$ in $G$. For any $C_1, C_2 \in \mathcal{C}$, 
$C_1 \cong K_{k+1}$ if and only if $C_2 \cong K_{k+1}$.
\end{lemma}

\begin{proof}
To the contrary, suppose that $C_1$ and $C_2$ are two closures of $v$ in $G$ such that $C_1\ncong K_{k+1}$ and $C_2\cong K_{k+1}$. 
Then by Lemma~\ref{lem:2.9}, we only need to consider the case $C_1\cap C_2=\{v\}$. By Proposition~\ref{prop:2.8}, $d_{C_1}(v)=k$ and the trivial edge-cut $Y$ incident to $v$ in $C_1$ is the unique $k$-edge-cut of $C_1$.

Note that $C_2$ is a closure of $v$, thus $C_2$ can be derived from $G$ by removing a series of $k$-edge-cuts. Let $G_1=G$, and let $G_{i+1}$ be the component of $G_{i}-X_{i}$ containing $C_2$ for $1\leq i\leq s-1$, with $G_s=C_2$, where $X_i$ is a $k$-edge-cut of $G_i$. Let $p$ be the least integer such that $X_p\cap E(C_1)\neq \emptyset$. Then $C_1$ is a subgraph of $G_p$ and $X_p$ is a $k$-edge-cut of $G_p$.

Without loss of generality, assume $e=xy\in X_p\cap E(C_1)$. Then $X_p$ must separate the vertices $x$ and $y$ in $C_1$. Note that $\kappa'(C_1)=k$, therefore separating $x$ and $y$ requires at least $k$ edges of $C_1$. Thus $X_p$ is also a $k$-edge-cut of $C_1$. By Proposition~\ref{prop:2.8}, $X_p=Y$. Let $G_{p+1}'$ be the component of $G_p-X_p$ that does not contain $G_{p+1}$, then we have $C_2\subseteq G_{p+1}$ and  $(C_1-v)\subseteq G_{p+1}'$. Therefore, both $G_{p+1}$ and $G_{p+1}'$ are kernels of $G$, implying they are $k$-edge-maximal.

Let $G'$ be the kernel of $G$ containing $C_2$ obtained from $G$ by recursively removing $(s-1)$ $k$-edge-cuts 
$$
X_1,\ldots,X_{p-1},X_{p+1},\ldots,X_s.
$$ 
Then $G'\in[C_2,G_{p+1}']_k$. Note that $C_2\cong K_{k+1}$ and $G_{p+1}'$ is $k$-edge-maximal. By Corollary~\ref{coro:2.4}, $G'$ is not $k$-edge-maximal, which contradicts the fact that  $G'$ is a kernel of $G$. This completes the proof.
\end{proof}

\begin{corollary}\label{coro:2.11}
Let $G$ be a $k$-edge-maximal graph, where $k\geq2$, and $H$ be a kernel of $G$. Then $H$ is the unique closure of itself.
\end{corollary}

\begin{proof}
Since $H$ is a kernel of $G$, $|V(H)|\geq k+1> 2$ and $H$ is a closure of itself. 
If $H\ncong K_{k+1}$, then $H$ is the unique closure of itself by Lemma~\ref{lem:2.9}. If $H\cong K_{k+1}$ and there is a closure $C_1$ of $H$ not isomorphic to $K_{k+1}$, then by Lemma~\ref{lem:2.9}, $C_1\ncong K_{k+1}$ is the unique closure of $H$, which contradicts the fact that $H$ is also a closure of itself. Therefore, each closure of $H$ is isomorphic to $K_{k+1}$. 
Note that $H\cong K_{k+1}$ is a subgraph of its closure. Therefore, for each closure $C$ of $H$ in $G$, we have
$$K_{k+1}\cong H\subseteq C\cong K_{k+1}.$$
Hence $C=H$, which means that $H$ is the unique closure of itself.
\end{proof}

By Lemma~\ref{lem:2.9} and Lemma~\ref{lem:2.10}, we fully characterize the properties of a closure as follows.

\begin{theorem}\label{thm:2.12}
Let $G$ be a $k$-edge-maximal graph, where $k\geq2$, and let $H$ be a subgraph of $G$. Denote by $\mathcal{C}$ the set of all closures of $H$ in $G$. Then one of the following holds:
\begin{itemize}
    \item[$(i)$] $C\cong K_{k+1}$ for each $C\in \mathcal{C}$.
    \item[$(ii)$] $C\ncong K_{k+1}$ for each $C\in\mathcal{C}$ when $|V(H)|=1$.
     \item[$(iii)$] $\mathcal{C}=\{C\}$ and $C\ncong K_{k+1}$ when $|V(H)|\geq2$.
\end{itemize}
\end{theorem}

Now, we list some examples to illustrate Theorem \ref{thm:2.12}.

(i) Let $G=H\vee (n-k)K_{1}$, where $H\cong K_k$ and $n\geq k+1$. Then every closure of a subgraph in $H$ is isomorphic to $K_{k+1}$. 

(ii) Consider $G_2$ in Figure~\ref{fig:fig0}. The vertex $v$ has two closures not isomorphic to $K_{3}$, which can be obtained by deleting exactly one of the trivial $2$-edge-cuts. 

(iii) The triangle containing the vertex $v$ of $G_2$ in Figure~\ref{fig:fig0} has a unique closure not isomorphic to $K_{3}$.

\section{Proofs of Theorems \ref{thm:main result}, \ref{thm:main result 2} and \ref{thm:main result 3}}

The aim of this section is to prove Theorems \ref{thm:main result}, \ref{thm:main result 2} and \ref{thm:main result 3}. Using these results, we obtain two corollaries.

\begin{lemma}\label{lem:3.1}
Let $G$ be a $k$-edge-maximal graph of order $n\geq k+1>2$.
If $G'\in [K_1,G]_k$, then $G'$ is also $k$-edge-maximal.
\end{lemma}

\begin{proof}
Let $G'$ be a graph obtained from $G$ by adding a new vertex $u$ together with $k$ edges connecting $u$ to $k$ distinct vertices $v_1,v_2,\ldots,v_k$ of $G$. It is clear that $\overline{\kappa}'(G')\leq k$. We need only show that for any edge $e\in E(G'^c)$, there exists a subgraph $H\subseteq G'+e$ such that $\kappa'(H)\geq k+1$. 

\textbf{Case 1.} $e \in E(G^c)$. Since $G$ is $k$-edge-maximal, there is a subgraph $H\subseteq G+e\subseteq G'+e$ such that $\kappa^{\prime}(H) \geq k+1$. 

\textbf{Case 2.} $e=uv_0\in E(G'^{c})$, where $v_0\in V(G)$. 
Denote $V_1=\{v_0,v_1,v_2,\ldots,v_k\}$. Let $C$ be a closure of $G[V_1]$ in $G$ and let $H=G'[V(C)\cup \{u\}]$. We will show that $\kappa'(H)\geq k+1$.

If $C\cong K_{k+1}$, then $H\cong K_{k+2}$ and $\kappa'(H)= k+1$ as desired. Next, we consider $C\ncong K_{k+1}$. 
To the contrary, suppose that $\kappa'(H)\leq k$ and $X$ is a minimum edge-cut of $H$ with $|X|\leq k$. Since $\kappa'(C)=\overline{\kappa}'(C)=k$, we have 
$$
X\cap \{uv_0,uv_1,\dots,uv_k\}=\emptyset.
$$
Then $X\subseteq E(C)$ is a $k$-edge-cut of $C$. As $C\not\cong K_{k+1}$ is a closure of $V_1$, $X$ separates the vertices of $V_1$ by Proposition \ref{prop:2.7}. However, the vertices of $V_1$ are connected by $k+1$ edges 
incident to $u$. Thus, $H-X$ is still connected, which yields a contradiction. Therefore, $\kappa'(H)\geq k+1$. This completes the proof.
\end{proof}

\begin{lemma}\label{lem:3.2}
Let $H_1$ and $H_2$ be two $k$-edge-maximal graphs with at least $k+2$ vertices, where $k\geq2$, and let $G$ be a graph in  $[H_1,H_2]_k$. Denote $V_i$ as the set of vertices in $H_i$ that are incident with the edges between $H_1$ and $H_2$ in $G$ for $i = 1, 2$. Then $G$ is $k$-edge-maximal if and only if there is a closure of $V_i$ in $H_i$ not isomorphic to $K_{k+1}$ for each $i=1,2$.
\end{lemma}

\begin{proof}
Let $G\in [H_1,H_2]_k$ be a graph obtained from the disjoint union of $H_1$ and $H_2$ by adding $k$ new edges $e_1,e_2,\ldots,e_k$ between $H_1$ and $H_2$ such that each $e_i=x_iy_i$ is incident to a vertex $x_i$ of $H_1$ and a vertex $y_i$ of $H_2$. Let $V_1$ be the multi-set $\{x_1,x_2,\ldots,x_k\}$ and $V_2$ be the multi-set $\{y_1,y_2,\ldots,y_k\}$.

We establish the sufficiency first. Suppose that there is a closure of $V_i$ in $H_i$ not isomorphic to $K_{k+1}$ for each $i=1,2$. Let $e\in E(G^c)$. It is clear that $\overline{\kappa}'(G)\leq k$. To prove that $G$ is $k$-edge-maximal, we need only find a subgraph $H\subseteq G+e$ with  $\kappa'(H)\geq k+1$.

\textbf{Case 1.} $e\in E(H_j^c)$ for some $j\in\{1,2\}$. Since $H_j$ is $k$-edge-maximal, there is a subgraph $H\subseteq H_j+e\subseteq G+e$ such that $\kappa'(H)\geq k+1$.

\textbf{Case 2.} $e=xy$, where $x\in V(H_1)$ and $y\in V(H_2)$. Let
$$
U_1=V_1\cup \{x\}=\{x_1,x_2,\ldots,x_k,x\} \text{ and } U_2=V_2\cup \{y\}=\{y_1,y_2,\ldots,y_k,y\}.
$$
Then $1\leq |U_i|\leq k+1$ for each $i\in\{1,2\}$ and at least one of $U_1$ and $U_2$ has at least two vertices. 
Without loss of generality, assume $2\leq |U_2|\leq k+1$. Let $C_i$ be a closure of $H_i[U_i]$ in $H_i$ for $i=1,2$. 
Thus, both $C_1$ and $C_2$ are $k$-edge-maximal graphs. For each $i=1,2$, 
if $C_i\cong K_{k+1}$, then $C_i$ is a closure of $V_i$, which contradicts there is a closure of $V_i$ in $H_i$ not isomorphic to $K_{k+1}$ by Theorem~\ref{thm:2.12}.
Thus $C_1$ and $C_2$ are not $K_{k+1}$'s and so $|V(C_i)|\geq k+2$ for each $i=1,2$.

\textbf{Case 2.1.} $|U_1|=1$. Then $U_1=\{x\}$ and $|U_2|=k+1$. 

Since $C_2$ is $k$-edge-maximal, $G[V(C_2)\cup \{x\}]$ is also $k$-edge-maximal by Lemma \ref{lem:3.1}. Hence, there exists a subgraph $H\subseteq G[V(C_2)\cup \{x\}]+e\subseteq G+e$ such that $\kappa'(H)\geq k+1$. 

\textbf{Case 2.2.} $|U_1|\geq 2$ and $|U_2|\geq 2$. We will find the subgraph $H\subseteq G+e$ such that $\kappa'(H)\geq k+1$. Consider the graph 
$$
H=G[V(C_1)\cup V(C_2)]+e\subseteq G+e. 
$$

To the contrary, suppose that $\kappa'(H)\leq k$ and $X$ is a minimum edge-cut of $H$ with $|X|\leq k$. Since $\kappa'(C_1)=\kappa'(C_2)=k$, then 
$$
X\cap \{e=xy,e_1,e_2,\ldots,e_k\}=\emptyset.
$$
Assume that $X\subseteq E(C_j)$ for some $j\in\{1,2\}$ such that $C_j-X$ is disconnected. Since $C_j\not\cong K_{k+1}$ is a closure of $U_j$, $X$ separates the vertices of $U_j$ by Proposition \ref{prop:2.7}. However, the vertices of $U_j$ are connected by $k+1$ edges $e=xy,e_1,e_2,\ldots,e_k$ and $C_{3-j}$. Thus, $H-X$ is still connected which yields a contradiction. Hence, $X\not\subseteq E(C_i)$ for each $i=1,2$, and we have both 
$X\cap E(C_1)\neq \emptyset$ and $X\cap E(C_2)\neq \emptyset$. Since $X$ is an edge-cut of $H$, $X$ must separates both $C_1$ and $C_2$. As $\kappa'(C_i) = k$ for each $i=1,2$, we have 
$$
|X|\geq k+k=2k>k, 
$$
which yields a contradiction. Thus, $\kappa'(H)\geq k+1$. We obtain the desired graph $H$ and so $G$ is $k$-edge-maximal.

Now, we verify the necessity. Suppose that $G$ is $k$-edge-maximal and $C_1$ is a closure of $H_1[V_1]$ with $C_1\cong K_{k+1}$. As $|V(H_1)|\geq k+2$ and $C_1$ is a kernel, $C_1$ can be obtained from $H_1$ by deleting a series of $k$-edge-cuts. Thus, $G[V(C_1)\cup V(H_2)]\in [C_1,H_2]_k$ can be viewed as the graph obtained from $G\in [H_1,H_2]_k$ by deleting a series of $k$-edge-cuts. By Lemma \ref{lemma-lai2}, $G[V(C_1)\cup V(H_2)]$ is also $k$-edge-maximal. However, $C_1\cong K_{k+1}$ implies that $G[V(C_1)\cup V(H_2)]\in [K_{k+1},H_2]_k$. By Corollary \ref{coro:2.4}, $G[V(C_1)\cup V(H_2)]$ is not $k$-edge-maximal, a contradiction. Similarly, $K_{k+1}$ cannot be a closure of $H_2[V_2]$. The result follows.
\end{proof}

\noindent \textbf{Proof of Theorem \ref{thm:main result}.}
Theorem \ref{thm:main result} (i) follows from Corollary \ref{coro:2.5} and Lemma \ref{lem:3.1}. Theorem \ref{thm:main result} (ii) follows from Lemma \ref{lemma-lai2} and Lemma \ref{lem:3.2}. \hfill $\Box$

\begin{corollary}\label{coro:3.3}
Let $H_1$ and $H_2$ be two $k$-edge-maximal graphs with at least $k+2$ vertices, and let $G\in [H_1,H_2]_k$ for $k\geq2$. Denote $V_i$ as the set of vertices in $H_i$ that are incident with the edges between $H_1$ and $H_2$ for each $i \in \{1, 2\}$. If either $V_i$ contains a pair of non-adjacent vertices or $V_i$ is not contained in any $(k+1)$-clique of $H_i$ for each $i \in \{1, 2\}$, then $G$ is $k$-edge-maximal.
\end{corollary}

\begin{proof}
If either $V_i$ contains a pair of non-adjacent vertices, or $V_i$ is not contained in any $(k+1)$-clique of $H_i$ for each $i \in \{1, 2\}$, then no closure of $H_i[V_i]$ can be $K_{k+1}$. By Theorem~\ref{thm:main result}, $G$ is $k$-edge-maximal.
\end{proof}

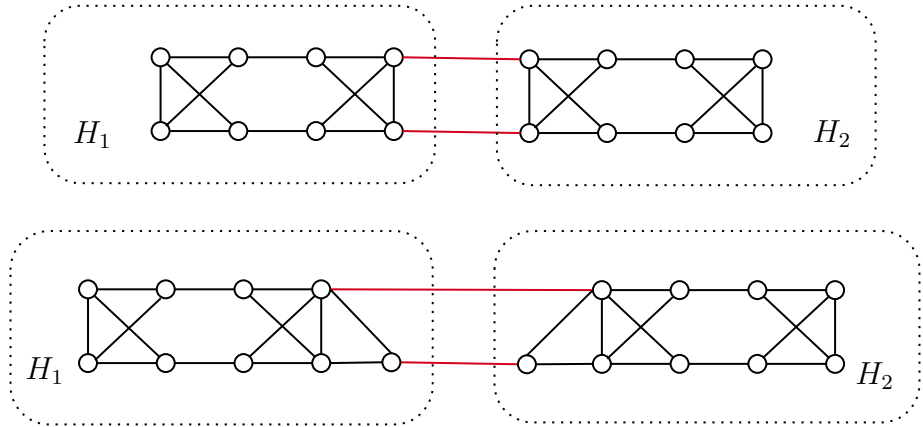
\begin{figure}[htbp]
    \centering

\tikzset{every picture/.style={line width=0.75pt}} 

\begin{tikzpicture}[x=0.75pt,y=0.75pt,yscale=-1,xscale=1]

\draw   (139.77,107.33) .. controls (139.77,104.85) and (141.78,102.83) .. (144.27,102.83) .. controls (146.75,102.83) and (148.77,104.85) .. (148.77,107.33) .. controls (148.77,109.82) and (146.75,111.83) .. (144.27,111.83) .. controls (141.78,111.83) and (139.77,109.82) .. (139.77,107.33) -- cycle ;
\draw   (139.77,144.33) .. controls (139.77,141.85) and (141.78,139.83) .. (144.27,139.83) .. controls (146.75,139.83) and (148.77,141.85) .. (148.77,144.33) .. controls (148.77,146.82) and (146.75,148.83) .. (144.27,148.83) .. controls (141.78,148.83) and (139.77,146.82) .. (139.77,144.33) -- cycle ;
\draw   (178.77,107.33) .. controls (178.77,104.85) and (180.78,102.83) .. (183.27,102.83) .. controls (185.75,102.83) and (187.77,104.85) .. (187.77,107.33) .. controls (187.77,109.82) and (185.75,111.83) .. (183.27,111.83) .. controls (180.78,111.83) and (178.77,109.82) .. (178.77,107.33) -- cycle ;
\draw   (178.77,144.33) .. controls (178.77,141.85) and (180.78,139.83) .. (183.27,139.83) .. controls (185.75,139.83) and (187.77,141.85) .. (187.77,144.33) .. controls (187.77,146.82) and (185.75,148.83) .. (183.27,148.83) .. controls (180.78,148.83) and (178.77,146.82) .. (178.77,144.33) -- cycle ;
\draw    (148.77,107.33) -- (178.77,107.33) ;
\draw    (148.77,144.33) -- (178.77,144.33) ;
\draw    (144.27,111.83) -- (144.27,139.83) ;
\draw    (147.42,110.5) -- (180.12,141.17) ;
\draw    (179.92,110.75) -- (146.92,141.75) ;
\draw   (217.77,107.33) .. controls (217.77,104.85) and (219.78,102.83) .. (222.27,102.83) .. controls (224.75,102.83) and (226.77,104.85) .. (226.77,107.33) .. controls (226.77,109.82) and (224.75,111.83) .. (222.27,111.83) .. controls (219.78,111.83) and (217.77,109.82) .. (217.77,107.33) -- cycle ;
\draw   (217.77,144.33) .. controls (217.77,141.85) and (219.78,139.83) .. (222.27,139.83) .. controls (224.75,139.83) and (226.77,141.85) .. (226.77,144.33) .. controls (226.77,146.82) and (224.75,148.83) .. (222.27,148.83) .. controls (219.78,148.83) and (217.77,146.82) .. (217.77,144.33) -- cycle ;
\draw   (256.77,107.33) .. controls (256.77,104.85) and (258.78,102.83) .. (261.27,102.83) .. controls (263.75,102.83) and (265.77,104.85) .. (265.77,107.33) .. controls (265.77,109.82) and (263.75,111.83) .. (261.27,111.83) .. controls (258.78,111.83) and (256.77,109.82) .. (256.77,107.33) -- cycle ;
\draw   (256.77,144.33) .. controls (256.77,141.85) and (258.78,139.83) .. (261.27,139.83) .. controls (263.75,139.83) and (265.77,141.85) .. (265.77,144.33) .. controls (265.77,146.82) and (263.75,148.83) .. (261.27,148.83) .. controls (258.78,148.83) and (256.77,146.82) .. (256.77,144.33) -- cycle ;
\draw    (226.77,107.33) -- (256.77,107.33) ;
\draw    (226.77,144.33) -- (256.77,144.33) ;
\draw    (225.42,110.5) -- (258.12,141.17) ;
\draw    (257.92,110.75) -- (225.42,141) ;
\draw    (261.27,111.83) -- (261.27,139.83) ;
\draw    (187.77,107.33) -- (217.77,107.33) ;
\draw    (187.77,144.33) -- (217.77,144.33) ;

\draw   (324.77,108.33) .. controls (324.77,105.85) and (326.78,103.83) .. (329.27,103.83) .. controls (331.75,103.83) and (333.77,105.85) .. (333.77,108.33) .. controls (333.77,110.82) and (331.75,112.83) .. (329.27,112.83) .. controls (326.78,112.83) and (324.77,110.82) .. (324.77,108.33) -- cycle ;
\draw   (324.77,145.33) .. controls (324.77,142.85) and (326.78,140.83) .. (329.27,140.83) .. controls (331.75,140.83) and (333.77,142.85) .. (333.77,145.33) .. controls (333.77,147.82) and (331.75,149.83) .. (329.27,149.83) .. controls (326.78,149.83) and (324.77,147.82) .. (324.77,145.33) -- cycle ;
\draw   (363.77,108.33) .. controls (363.77,105.85) and (365.78,103.83) .. (368.27,103.83) .. controls (370.75,103.83) and (372.77,105.85) .. (372.77,108.33) .. controls (372.77,110.82) and (370.75,112.83) .. (368.27,112.83) .. controls (365.78,112.83) and (363.77,110.82) .. (363.77,108.33) -- cycle ;
\draw   (363.77,145.33) .. controls (363.77,142.85) and (365.78,140.83) .. (368.27,140.83) .. controls (370.75,140.83) and (372.77,142.85) .. (372.77,145.33) .. controls (372.77,147.82) and (370.75,149.83) .. (368.27,149.83) .. controls (365.78,149.83) and (363.77,147.82) .. (363.77,145.33) -- cycle ;
\draw    (333.77,108.33) -- (363.77,108.33) ;
\draw    (333.77,145.33) -- (363.77,145.33) ;
\draw    (329.27,112.83) -- (329.27,140.83) ;
\draw    (332.42,111.5) -- (365.12,142.17) ;
\draw    (364.92,111.75) -- (331.92,142.75) ;
\draw   (402.77,108.33) .. controls (402.77,105.85) and (404.78,103.83) .. (407.27,103.83) .. controls (409.75,103.83) and (411.77,105.85) .. (411.77,108.33) .. controls (411.77,110.82) and (409.75,112.83) .. (407.27,112.83) .. controls (404.78,112.83) and (402.77,110.82) .. (402.77,108.33) -- cycle ;
\draw   (402.77,145.33) .. controls (402.77,142.85) and (404.78,140.83) .. (407.27,140.83) .. controls (409.75,140.83) and (411.77,142.85) .. (411.77,145.33) .. controls (411.77,147.82) and (409.75,149.83) .. (407.27,149.83) .. controls (404.78,149.83) and (402.77,147.82) .. (402.77,145.33) -- cycle ;
\draw   (441.77,108.33) .. controls (441.77,105.85) and (443.78,103.83) .. (446.27,103.83) .. controls (448.75,103.83) and (450.77,105.85) .. (450.77,108.33) .. controls (450.77,110.82) and (448.75,112.83) .. (446.27,112.83) .. controls (443.78,112.83) and (441.77,110.82) .. (441.77,108.33) -- cycle ;
\draw   (441.77,145.33) .. controls (441.77,142.85) and (443.78,140.83) .. (446.27,140.83) .. controls (448.75,140.83) and (450.77,142.85) .. (450.77,145.33) .. controls (450.77,147.82) and (448.75,149.83) .. (446.27,149.83) .. controls (443.78,149.83) and (441.77,147.82) .. (441.77,145.33) -- cycle ;
\draw    (411.77,108.33) -- (441.77,108.33) ;
\draw    (411.77,145.33) -- (441.77,145.33) ;
\draw    (410.42,111.5) -- (443.12,142.17) ;
\draw    (442.92,111.75) -- (410.42,142) ;
\draw    (446.27,112.83) -- (446.27,140.83) ;
\draw    (372.77,108.33) -- (402.77,108.33) ;
\draw    (372.77,145.33) -- (402.77,145.33) ;

\draw [color={rgb, 255:red, 208; green, 2; blue, 27 }  ,draw opacity=1 ]   (265.77,107.33) -- (324.77,108.33) ;
\draw [color={rgb, 255:red, 208; green, 2; blue, 27 }  ,draw opacity=1 ]   (265.77,144.33) -- (324.77,145.33) ;
\draw  [dash pattern={on 0.84pt off 2.51pt}] (86,99.3) .. controls (86,89.47) and (93.97,81.5) .. (103.8,81.5) -- (264.2,81.5) .. controls (274.03,81.5) and (282,89.47) .. (282,99.3) -- (282,152.7) .. controls (282,162.53) and (274.03,170.5) .. (264.2,170.5) -- (103.8,170.5) .. controls (93.97,170.5) and (86,162.53) .. (86,152.7) -- cycle ;
\draw  [dash pattern={on 0.84pt off 2.51pt}] (313,99.5) .. controls (313,89.56) and (321.06,81.5) .. (331,81.5) -- (485,81.5) .. controls (494.94,81.5) and (503,89.56) .. (503,99.5) -- (503,153.5) .. controls (503,163.44) and (494.94,171.5) .. (485,171.5) -- (331,171.5) .. controls (321.06,171.5) and (313,163.44) .. (313,153.5) -- cycle ;
\draw   (103.37,223.73) .. controls (103.37,221.25) and (105.38,219.23) .. (107.87,219.23) .. controls (110.35,219.23) and (112.37,221.25) .. (112.37,223.73) .. controls (112.37,226.22) and (110.35,228.23) .. (107.87,228.23) .. controls (105.38,228.23) and (103.37,226.22) .. (103.37,223.73) -- cycle ;
\draw   (103.37,260.73) .. controls (103.37,258.25) and (105.38,256.23) .. (107.87,256.23) .. controls (110.35,256.23) and (112.37,258.25) .. (112.37,260.73) .. controls (112.37,263.22) and (110.35,265.23) .. (107.87,265.23) .. controls (105.38,265.23) and (103.37,263.22) .. (103.37,260.73) -- cycle ;
\draw   (142.37,223.73) .. controls (142.37,221.25) and (144.38,219.23) .. (146.87,219.23) .. controls (149.35,219.23) and (151.37,221.25) .. (151.37,223.73) .. controls (151.37,226.22) and (149.35,228.23) .. (146.87,228.23) .. controls (144.38,228.23) and (142.37,226.22) .. (142.37,223.73) -- cycle ;
\draw   (142.37,260.73) .. controls (142.37,258.25) and (144.38,256.23) .. (146.87,256.23) .. controls (149.35,256.23) and (151.37,258.25) .. (151.37,260.73) .. controls (151.37,263.22) and (149.35,265.23) .. (146.87,265.23) .. controls (144.38,265.23) and (142.37,263.22) .. (142.37,260.73) -- cycle ;
\draw    (112.37,223.73) -- (142.37,223.73) ;
\draw    (112.37,260.73) -- (142.37,260.73) ;
\draw    (107.87,228.23) -- (107.87,256.23) ;
\draw    (111.02,226.9) -- (143.72,257.57) ;
\draw    (144.37,228.23) -- (111.37,259.23) ;
\draw   (181.37,223.73) .. controls (181.37,221.25) and (183.38,219.23) .. (185.87,219.23) .. controls (188.35,219.23) and (190.37,221.25) .. (190.37,223.73) .. controls (190.37,226.22) and (188.35,228.23) .. (185.87,228.23) .. controls (183.38,228.23) and (181.37,226.22) .. (181.37,223.73) -- cycle ;
\draw   (181.37,260.73) .. controls (181.37,258.25) and (183.38,256.23) .. (185.87,256.23) .. controls (188.35,256.23) and (190.37,258.25) .. (190.37,260.73) .. controls (190.37,263.22) and (188.35,265.23) .. (185.87,265.23) .. controls (183.38,265.23) and (181.37,263.22) .. (181.37,260.73) -- cycle ;
\draw   (220.37,223.73) .. controls (220.37,221.25) and (222.38,219.23) .. (224.87,219.23) .. controls (227.35,219.23) and (229.37,221.25) .. (229.37,223.73) .. controls (229.37,226.22) and (227.35,228.23) .. (224.87,228.23) .. controls (222.38,228.23) and (220.37,226.22) .. (220.37,223.73) -- cycle ;
\draw   (220.37,260.73) .. controls (220.37,258.25) and (222.38,256.23) .. (224.87,256.23) .. controls (227.35,256.23) and (229.37,258.25) .. (229.37,260.73) .. controls (229.37,263.22) and (227.35,265.23) .. (224.87,265.23) .. controls (222.38,265.23) and (220.37,263.22) .. (220.37,260.73) -- cycle ;
\draw    (190.37,223.73) -- (220.37,223.73) ;
\draw    (190.37,260.73) -- (220.37,260.73) ;
\draw    (189.02,226.9) -- (221.72,257.57) ;
\draw    (221.52,227.15) -- (189.02,257.4) ;
\draw    (224.87,228.23) -- (224.87,256.23) ;
\draw    (151.37,223.73) -- (181.37,223.73) ;
\draw    (151.37,260.73) -- (181.37,260.73) ;

\draw   (361.17,224.13) .. controls (361.17,221.65) and (363.18,219.63) .. (365.67,219.63) .. controls (368.15,219.63) and (370.17,221.65) .. (370.17,224.13) .. controls (370.17,226.62) and (368.15,228.63) .. (365.67,228.63) .. controls (363.18,228.63) and (361.17,226.62) .. (361.17,224.13) -- cycle ;
\draw   (361.17,261.13) .. controls (361.17,258.65) and (363.18,256.63) .. (365.67,256.63) .. controls (368.15,256.63) and (370.17,258.65) .. (370.17,261.13) .. controls (370.17,263.62) and (368.15,265.63) .. (365.67,265.63) .. controls (363.18,265.63) and (361.17,263.62) .. (361.17,261.13) -- cycle ;
\draw   (400.17,224.13) .. controls (400.17,221.65) and (402.18,219.63) .. (404.67,219.63) .. controls (407.15,219.63) and (409.17,221.65) .. (409.17,224.13) .. controls (409.17,226.62) and (407.15,228.63) .. (404.67,228.63) .. controls (402.18,228.63) and (400.17,226.62) .. (400.17,224.13) -- cycle ;
\draw   (400.17,261.13) .. controls (400.17,258.65) and (402.18,256.63) .. (404.67,256.63) .. controls (407.15,256.63) and (409.17,258.65) .. (409.17,261.13) .. controls (409.17,263.62) and (407.15,265.63) .. (404.67,265.63) .. controls (402.18,265.63) and (400.17,263.62) .. (400.17,261.13) -- cycle ;
\draw    (370.17,224.13) -- (400.17,224.13) ;
\draw    (370.17,261.13) -- (400.17,261.13) ;
\draw    (365.67,228.63) -- (365.67,256.63) ;
\draw    (368.82,227.3) -- (401.52,257.97) ;
\draw    (401.32,227.55) -- (368.32,258.55) ;
\draw   (439.17,224.13) .. controls (439.17,221.65) and (441.18,219.63) .. (443.67,219.63) .. controls (446.15,219.63) and (448.17,221.65) .. (448.17,224.13) .. controls (448.17,226.62) and (446.15,228.63) .. (443.67,228.63) .. controls (441.18,228.63) and (439.17,226.62) .. (439.17,224.13) -- cycle ;
\draw   (439.17,261.13) .. controls (439.17,258.65) and (441.18,256.63) .. (443.67,256.63) .. controls (446.15,256.63) and (448.17,258.65) .. (448.17,261.13) .. controls (448.17,263.62) and (446.15,265.63) .. (443.67,265.63) .. controls (441.18,265.63) and (439.17,263.62) .. (439.17,261.13) -- cycle ;
\draw   (478.17,224.13) .. controls (478.17,221.65) and (480.18,219.63) .. (482.67,219.63) .. controls (485.15,219.63) and (487.17,221.65) .. (487.17,224.13) .. controls (487.17,226.62) and (485.15,228.63) .. (482.67,228.63) .. controls (480.18,228.63) and (478.17,226.62) .. (478.17,224.13) -- cycle ;
\draw   (478.17,261.13) .. controls (478.17,258.65) and (480.18,256.63) .. (482.67,256.63) .. controls (485.15,256.63) and (487.17,258.65) .. (487.17,261.13) .. controls (487.17,263.62) and (485.15,265.63) .. (482.67,265.63) .. controls (480.18,265.63) and (478.17,263.62) .. (478.17,261.13) -- cycle ;
\draw    (448.17,224.13) -- (478.17,224.13) ;
\draw    (448.17,261.13) -- (478.17,261.13) ;
\draw    (446.82,227.3) -- (479.52,257.97) ;
\draw    (479.32,227.55) -- (446.82,257.8) ;
\draw    (482.67,228.63) -- (482.67,256.63) ;
\draw    (409.17,224.13) -- (439.17,224.13) ;
\draw    (409.17,261.13) -- (439.17,261.13) ;

\draw [color={rgb, 255:red, 208; green, 2; blue, 27 }  ,draw opacity=1 ]   (229.37,223.73) -- (361.17,224.13) ;
\draw [color={rgb, 255:red, 208; green, 2; blue, 27 }  ,draw opacity=1 ]   (264.57,260.33) -- (323.57,261.33) ;
\draw  [dash pattern={on 0.84pt off 2.51pt}] (69,213.9) .. controls (69,203.19) and (77.69,194.5) .. (88.4,194.5) -- (261.4,194.5) .. controls (272.11,194.5) and (280.8,203.19) .. (280.8,213.9) -- (280.8,272.1) .. controls (280.8,282.81) and (272.11,291.5) .. (261.4,291.5) -- (88.4,291.5) .. controls (77.69,291.5) and (69,282.81) .. (69,272.1) -- cycle ;
\draw  [dash pattern={on 0.84pt off 2.51pt}] (311.8,213.7) .. controls (311.8,203.1) and (320.4,194.5) .. (331,194.5) -- (505.8,194.5) .. controls (516.4,194.5) and (525,203.1) .. (525,213.7) -- (525,271.3) .. controls (525,281.9) and (516.4,290.5) .. (505.8,290.5) -- (331,290.5) .. controls (320.4,290.5) and (311.8,281.9) .. (311.8,271.3) -- cycle ;
\draw   (255.57,260.33) .. controls (255.57,257.85) and (257.58,255.83) .. (260.07,255.83) .. controls (262.55,255.83) and (264.57,257.85) .. (264.57,260.33) .. controls (264.57,262.82) and (262.55,264.83) .. (260.07,264.83) .. controls (257.58,264.83) and (255.57,262.82) .. (255.57,260.33) -- cycle ;
\draw   (323.57,261.33) .. controls (323.57,258.85) and (325.58,256.83) .. (328.07,256.83) .. controls (330.55,256.83) and (332.57,258.85) .. (332.57,261.33) .. controls (332.57,263.82) and (330.55,265.83) .. (328.07,265.83) .. controls (325.58,265.83) and (323.57,263.82) .. (323.57,261.33) -- cycle ;
\draw    (229.37,223.73) -- (260.07,255.83) ;
\draw    (229.37,260.73) -- (255.57,260.33) ;
\draw    (361.17,224.13) -- (328.07,256.83) ;
\draw    (332.57,261.33) -- (361.17,261.13) ;

\draw (99,137.9) node [anchor=north west][inner sep=0.75pt]    {$H_{1}$};
\draw (470,137.4) node [anchor=north west][inner sep=0.75pt]    {$H_{2}$};
\draw (75,256.5) node [anchor=north west][inner sep=0.75pt]    {$H_{1}$};
\draw (491.4,259) node [anchor=north west][inner sep=0.75pt]    {$H_{2}$};

\end{tikzpicture}

    \caption{Two non-$2$-edge-maximal graphs}
    \label{fig:fig1}
\end{figure}

The graph $G_1$ in Figure \ref{fig:fig0} and the graphs $[H_1,H_2]_2$ in Figure \ref{fig:fig1} violate the necessary condition in Theorem \ref{thm:main result}. It is easy to check that they are not $2$-edge-maximal by deleting a series of $2$-edge-cuts.

\bigskip
\noindent \textbf{Proof of Theorem \ref{thm:main result 2}.}
To begin with, we verify the sufficiency. Suppose $1\leq r\leq\left\lfloor \frac{n}{k+2}\right\rfloor$. We construct a $k$-edge-maximal graph $G$ of order $n$ with $ (n-1)k-\binom{k}{2}r$ edges as follows.

For $r=1$, $G=K_{k}\vee (n-k)K_1$ is the desired graph.
For $r\geq2$, let $H_1\cong K_k \vee (n+2-(k+2)r)K_1$ and $H_i\cong K_k \vee 2K_1$ for each $i\in\{2,\ldots,r\}$. There are two vertices $u_i$ and $v_i$ of degree $k$ in $H_i$ for each $i\in \{1,2,\ldots,r\}$. Clearly, $u_i$ is not adjacent to $v_i$. Let $G$ be a graph obtained from $\bigcup_{i=1}^r H_i$ by adding $k$ edges between $H_1$  and $H_i$ for each $i\in\{2,3,\ldots,r\}$, such that $\{u_1u_i,v_1v_i: 2\leq i\leq r\}\subseteq E(G)$. By Corollary~\ref{coro:3.3}, $G$ is $k$-edge-maximal with
$$
|E(G)|= \binom{k}{2}r+k(n-kr)+k(r-1)=(n-1)k-\binom{k}{2}r.
$$

Now, we verify the necessity. The conclusion will be proved by induction on $n$. Let $G$ be a $k$-edge-maximal graph of order $n$ with $m$ edges.

If $n\leq 2k+3$, then $\left\lfloor\frac{n}{k+2}\right\rfloor=1$.
By Theorem~\ref{thm:main result} all $k$-edge-cuts of $G$ are trivial, and there exists a sequence of graphs $G_1,G_2,\ldots,G_{n-k}$, such that
$G_1=G$, $G_{n-k}\cong K_{k+1}$ and $G_{i+1}=G_{i}-v_{i}$ for $1\leq i\leq n-k-1$, where $v_i$ is a vertex of degree $k$ in $G_i$.  Thus 
$$
m=\binom{k+1}{2}+(n-k-1)k=(n-1)k-\binom{k}{2}.
$$
The result follows.

Assume that $n\geq 2k+4$ and the conclusion holds for all $k$-edge-maximal graphs of order less than $n$.

\textbf{Case 1.} $G$ has a vertex $v$ of degree $k$. Then $G-v$ is also $k$-edge-maximal by Corollary \ref{coro:2.5}. According to the induction hypothesis, 
$$
|E(G-v)|\in\left\{(n-2)k-\binom{k}{2}r: 1\leq r\leq\left\lfloor \frac{n-1}{k+2}\right\rfloor\text{ and } r\in \mathbb{N}\right\}.
$$
Therefore,
$$
\begin{aligned}
m=|E(G-v)|+k&\in\left\{(n-1)k-\binom{k}{2}r: 1\leq r\leq\left\lfloor \frac{n-1}{k+2}\right\rfloor\text{ and } r\in \mathbb{N}\right\}\\
&\subseteq\left\{(n-1)k-\binom{k}{2}r: 1\leq r\leq\left\lfloor \frac{n}{k+2}\right\rfloor\text{ and } r\in \mathbb{N}\right\}.
\end{aligned}
$$

\textbf{Case 2.} There is a $k$-edge-cut $X$ of $G$, such that $G-X$ is a disjoint union of two $k$-edge-maximal graphs $H_1$ and $H_2$. Denote $|V(H_i)|=n_i$ for each $i\in \{1,2\}$.

By the induction hypothesis, for each $i\in \{1,2\}$, there exists an integer $r_i$ with $1\leq r_i\leq\left\lfloor \frac{n_i}{k+2}\right\rfloor$ such that
$$
|E(H_i)|=(n_i-1)k-\binom{k}{2}r_i.
$$
Therefore,
$$
m=|E(H_1)|+|E(H_2)|+k=(n-1)k-\binom{k}{2}(r_1+r_2).
$$
It suffices to show  $r_1+r_2\leq\left\lfloor \frac{n}{k+2}\right\rfloor=\left\lfloor\frac{n_1+n_2}{k+2}\right\rfloor$, which follows by the fact
$$
\left\lfloor \frac{n_1}{k+2}\right\rfloor+\left\lfloor \frac{n_2}{k+2}\right\rfloor\leq\left\lfloor \frac{n_1+n_2}{k+2}\right\rfloor.
$$
\hfill $\Box$

\begin{corollary}[Mader \cite{mad71}]\label{theorem-mad71}
Let $k$ be a positive integer. If $G$ is a $k$-edge-maximal graph with $n \geq k+1$ vertices, then
$$
|E(G)| \leq F(n, k)=(n-1) k-\binom{k}{2}.
$$
Furthermore, $G \in \mathcal{G}(F ; n, k)$ if and only if $G\cong K_{k+1}$ or $G$ has a vertex $v$ of degree $k$ such that $G-v \in \mathcal{G}(F ; n-1, k)$.
\end{corollary}

\begin{proof}
Let $G$ be a $k$-edge-maximal graph with $n \geq k+1$ vertices.
If $k=1$ or $n=k+1$, then the result holds.
Next we assume $n>k+1>2$. By Theorem~\ref{thm:main result 2}, we have
$$
|E(G)|\leq (n-1)k-\binom{k}{2}.
$$

For the remaining result,
we establish the necessity first. Let $G \in \mathcal{G}(F ; n, k)$, then $|E(G)|=(n-1)k-\binom{k}{2}$. 
If $n=k+1$, $K_{k+1}$ is the unique $k$-edge-maximal graph. Now, assume  $n\geq k+2$. We first claim that $G$ has a vertex of degree $k$. To the contrary, suppose that
$G$ has no vertex of degree $k$. Then by Lemma \ref{lemma-lai2}, there is a $k$-edge-cut $X$ such that the components of  $G-X$ are two $k$-edge-maximal graphs $H_1$ and $H_2$ with $|V(H_i)|=n_i\geq k+2$ for each $i\in\{1,2\}$. By Theorem \ref{thm:main result 2}, for each $i=1,2$, there exists some $r_i\geq 1$ such that
$$
|E(H_i)|=(n_i-1)k-\binom{k}{2}r_i. 
$$
Therefore, 
$$
\begin{aligned}
|E(G)|=&|E(H_1)|+|E(H_2)|+k=(n-1)k-\binom{k}{2}(r_1+r_2)\\
 \leq & (n-1)k-2\binom{k}{2} \\
 <& (n-1)k-\binom{k}{2}.
\end{aligned}
$$
The last inequality holds as $k\geq2$. This inequality contradicts 
$|E(G)|=(n-1)k-\binom{k}{2}.$

Let $v$ be a vertex of degree $k$ in $G$. 
Then $G-v$ is $k$-edge-maximal by Corollary \ref{coro:2.5} and so
$$
|E(G-v)|=|E(G)|-k= (n-2) k-\binom{k}{2},
$$
which implies $G-v \in \mathcal{G}(F ; n-1, k)$.

Now, we verify the sufficiency. If $G=K_{k+1}$, then $G\in \mathcal{G}(F;n,k)$.
If $|V(G)|\geq k+2$ and $G$ has a vertex $v$ of degree $k$ satisfying $G-v \in \mathcal{G}(F ; n-1, k)$, then $G$ is $k$-edge-maximal by Lemma~\ref{lem:3.1}, and
$$
|E(G)|=|E(G-v)|+k=(n-2)k-\binom{k}{2}+k=(n-1)k-\binom{k}{2}.
$$
Hence, $G\in \mathcal{G}(F ; n, k)$.
\end{proof}

\begin{corollary}[Lai \cite{lai90}]
Let $G$ be a $k$-edge-maximal graph of order $n> k+1\geq 2$. Then
\begin{equation}\label{eq:eq2}
|E(G)|\geq f(n,k)=(n-1)k-\binom{k}{2}\left\lfloor \frac{n}{k+2}\right\rfloor.
\end{equation}
Furthermore, $G\in \mathcal{G}(f;n,k)$ if and only if $G\cong K_k\vee 2K_1$, or there is a $k$-edge-cut $X$ such that $G-X$ is a disjoint union of two graphs $H_1$ and $H_2$ with $\max\{|V(H_1)|,|V(H_2)|\}\geq k+2$, where either $H_i\in \mathcal{G}(f;n_i,k)$ with $n_i\geq k+2$ or $H_i\cong K_1$ satisfying
\begin{equation}\label{eq:eq3}
    \left\lfloor \frac{n_1}{k+2}\right\rfloor+\left\lfloor \frac{n_2}{k+2}\right\rfloor=\left\lfloor \frac{n}{k+2}\right\rfloor.
    \end{equation}
\end{corollary}

\begin{proof}
By Theorem \ref{thm:main result 2}, inequality \eqref{eq:eq2} holds as $r\leq \left\lfloor \frac{n}{k+2}\right\rfloor$. 

Let $G$ be a $k$-edge-maximal graph of order $n\geq k+2$. If $n=k+2$, then $G\cong K_k\vee 2K_1$, and $G\in \mathcal{G}(f;n,k)$. 
Now suppose that $n>k+2$, and let $X$ be a $k$-edge-cut of $G$ such that $G-X$ consists of a disjoint union of $H_1$ and $H_2$.

\textbf{Case 1.} For each $i\in \{1,2\}$, $|V(H_i)|=n_i>1$. Then both $H_1$ and $H_2$ are $k$-edge-maximal graphs by Theorem \ref{thm:main result} (ii). Thus, by Theorem~\ref{thm:main result 2}, there exists some $r_i\leq \left\lfloor \frac{n_i}{k+2}\right\rfloor$ such that
$$
 |E(H_i)|=(n_i-1)k-\binom{k}{2}r_i,
$$
for each $i\in \{1,2\}$. Thus, 
$$
\begin{aligned}
|E(G)|=&|E(H_1)|+|E(H_2)|+k=(n-1)k-\binom{k}{2}(r_1+r_2)\\
 \geq & (n-1)k-\binom{k}{2}\left(\left\lfloor \frac{n_1}{k+2}\right\rfloor+\left\lfloor \frac{n_2}{k+2}\right\rfloor\right)\\
 \geq & (n-1)k-\binom{k}{2} \left\lfloor \frac{n}{k+2}\right\rfloor.
\end{aligned}
$$
Therefore, $G\in\mathcal{G}(f;n,k)$ if and only if equality \eqref{eq:eq3} holds and
the following two equalities hold:
$$
\begin{aligned}
 &|E(H_1)|=(n_1-1)k-\binom{k}{2}\left\lfloor \frac{n_1}{k+2}\right\rfloor, \\
 &|E(H_2)|=(n_2-1)k-\binom{k}{2}\left\lfloor \frac{n_2}{k+2}\right\rfloor.
\end{aligned}
$$
The two equalities stated above hold if and only if $H_i\in\mathcal{G}(f;n_i,k)$ for each $i\in \{1,2\}$.

\textbf{Case 2.} $|V(H_i)|=n_i=1$ for some $i=1,2$, without loss of generality, $n_1=1$ and $n_2=n-1$. Then $H_2$ is a $k$-edge-maximal graph by Theorem \ref{thm:main result} (i).  By Theorem~\ref{thm:main result 2}, there exists $r_2\leq \left\lfloor \frac{n_2}{k+2}\right\rfloor=\left\lfloor\frac{n-1}{k+2}\right\rfloor$ such that 
$$
|E(H_2)|=(n_2-1)k-\binom{k}{2}r_2.
$$ 
Thus, 
$$
\begin{aligned}
|E(G)| =|E(H_2)|+k &= (n-1)k-\binom{k}{2}r_2   \\
&\geq  (n-1)k-\binom{k}{2}\left\lfloor \frac{n-1}{k+2}\right\rfloor\\
 & \geq (n-1)k-\binom{k}{2}\left\lfloor \frac{n}{k+2}\right\rfloor.
\end{aligned}
$$
Therefore, $G\in\mathcal{G}(f;n,k)$ if and only if equality \eqref{eq:eq3} and the following equality holds
$$
 |E(H_2)|=(n-2)k-\binom{k}{2}\left\lfloor \frac{n-1}{k+2}\right\rfloor.
$$
The equality above holds if and only if 
$H_2\in\mathcal{G}(f;n-1,k)$.
\end{proof}

\noindent \textbf{Proof of Theorem \ref{thm:main result 3}.}
We establish the necessity first. Suppose $G$ is a $k$-edge-maximal graph of order $n$ with 
\begin{equation}\label{eq:eq4}
m=(n-1)k - \binom{k}{2}r
\end{equation}
edges, where $n\geq k+2$, $1 \leq r \leq \left\lfloor \frac{n}{k+2} \right\rfloor $ and $ r \in \mathbb{N}$. We proceed by induction on $n$. 

If $n=k+2$, then $r=1$ and so $G\cong K_k \vee 2K_1$. In this case, no proper $k$-edge-joins are required.

If $n>k+2$, assume that the conclusion holds for all $k$-edge-maximal graphs of order less than $n$.

\textbf{Case 1.} $G$ has a vertex $v$ of degree $k$. By Corollary \ref{coro:2.5}, $G-v$ is also $k$-edge-maximal with 
$$
|E(G-v)|=(n-2)k-\binom{k}{2}r.
$$
By the induction hypothesis, $ G-v $ can be obtained from 
$$
r(K_k \vee 2K_1) \;\cup\; (n-1-(k+2)r)K_1
$$  
by a sequence of proper $ k $-edge-joins.
Therefore, $ G $ can be obtained from  
$$
r(K_k \vee 2K_1) \;\cup\; (n-(k+2)r)K_1
$$  
by a sequence of proper $ k $-edge-joins.

\textbf{Case 2.} There is a $k$-edge-cut $X$ of $G$, such that $G-X$ is a disjoint union of two $k$-edge-maximal graphs $H_1$ and $H_2$. Denote $|V(H_i)|=n_i$ for each $i=1,2$.

By Theorem~\ref{thm:main result 2}, there are integers $r_1$ and $r_2$ with $1\leq r_1\leq\left\lfloor \frac{n_1}{k+2}\right\rfloor$ and $1\leq r_2\leq\left\lfloor \frac{n_2}{k+2}\right\rfloor$ such that
$$
|E(H_1)|=(n_1-1)k-\binom{k}{2}r_1,\quad |E(H_2)|=(n_2-1)k-\binom{k}{2}r_2.
$$
By the induction hypothesis, for each $i=1,2$, the graph $ H_i$ can be obtained from  
$$
r_i(K_k \vee 2K_1) \;\cup\; (n_i-(k+2)r_i)K_1
$$  
by a sequence of proper $ k $-edge-joins.
Therefore, $ G$ can be obtained from  
$$
(r_1+r_2)(K_k \vee 2K_1) \;\cup\; (n_1+n_2-(k+2)(r_1+r_2))K_1
$$  
by a sequence of proper $ k $-edge-joins.
Note that $n_1+n_2=n$, and
\begin{equation}\label{eq:eq5}
 m=|E(H_1)|+|E(H_2)|+k=(n-1)k-\binom{k}{2}(r_1+r_2).   
\end{equation}
Combing \eqref{eq:eq4} with \eqref{eq:eq5}, 
we have $r_1+r_2=r$, and the result follows.

Now, we verify the sufficiency. Since $K_k\vee 2K_1$ is a $k$-edge-maximal graph with $k+2$ vertices, and every $k$-edge-join is proper, we can get a $k$-edge-maximal graph $G$. To count the number of edges in $G$, there are $\binom{k}{2}r+2kr$ edges in the disjoint union $r(K_k \vee 2K_1) \cup (n-(k+2)r)K_1$ and $k(n-(k+2)r+r-1)$ edges induced by $(n-(k+2)r+r-1)$-time  $k$-edge-join. Therefore,

$$
\begin{aligned}
|E(G)|&=\binom{k}{2}r+2kr+k(n-(k+2)r+r-1)\\
&=(n-1)k-\binom{k}{2}r,    
\end{aligned}
$$
which completes the proof.
\hfill $\Box$

\section{Concluding remark}

Given an integer $k\geq 2$, we establish a necessary and sufficient condition for a graph to be $k$-edge-maximal. This condition implies a method for determining whether a graph is $k$-edge-maximal. In fact, if $X$ is a $k$-edge-cut of a $k$-edge-maximal graph $G$, then both components $H_1$ and $H_2$ of $G-X$ are necessarily $k$-edge-maximal, and $X$ is a proper $k$-edge-join of $H_1$ and $H_2$. Moreover, the edge spectrum of a $k$-edge-maximal graph is presented in this paper. 

Anderson, Lai, Lin and Xu \cite{allx17} initially investigated the $k$-arc-maximal digraphs. The upper and lower bounds on the arc number of a $k$-arc-maximal digraph are established in \cite{allx17} and \cite{lflx16}, respectively.

\begin{theorem}{\rm (\cite{allx17} and \cite{lflx16})}
Let $n$ and $k$ be positive integers with $n \geq k+1$. If $D$ is a $k$-arc-maximal digraph of order $n$ with $m$ arcs, then
$$
\binom{ n}{2}+(n-1) k+\left\lfloor\frac{n}{k+2}\right\rfloor\left(1+2 k-\binom{k+2}{2}\right)\leq m \leq k(2n-k-1)+\binom{n-k}{2}.
$$
Furthermore, the upper and lower bounds are best possible.
\end{theorem}

It is worthwhile to investigate the necessary and sufficient condition for a digraph to be $k$-arc-maximal and characterize all possible values of the arc numbers for $k$-arc-maximal digraphs of a given order. We will further investigate these problems in subsequent work.

\section*{Declaration of competing interest}

There is no competing interest.

\section*{Data availability}

No data was used for the research described in the article.

\section*{Acknowledgement}
The research of Zheng-Jiang Xia is supported by
Key Projects in Natural Science Research of Anhui Provincial Department of Education (No. 2023AH050268). The research of Hong-Jian Lai is supported by Natural Science Foundation of China (No. 12471333).
The research of Zhen-Mu Hong is supported by Natural Science Foundation of China (No. 12371338) and Outstanding Youth Scientific Research Projects of Anhui Provincial Department of Education (No. 2022AH030073).

\end{document}